\documentclass[12pt]{amsart}
\usepackage{amscd}
\usepackage{amsfonts}
\usepackage[dvipsnames]{xcolor}
\usepackage{amsmath,amssymb,amsthm} 
\usepackage{ulem}  
\usepackage{algorithm}
\usepackage{url}
\usepackage{cite}
\usepackage{tikz}
\usepackage[all]{xy}
\usepackage{enumitem} 
\newtheorem{thm}{Theorem}[section]

\theoremstyle{definition}
\newtheorem{ex}[thm]{Example}
\newtheorem{lem}[thm]{Lemma}
\newtheorem{prop}[thm]{Proposition}
\theoremstyle{definition}
\newtheorem{defn}[thm]{Definition}

\numberwithin{equation}{section}

\newtheorem{claim}[thm]{Claim}
\def\P{\mathbb{P}} 
\def\Z{\mathbb{Z}} 
\def\C{\mathbb{C}} 
\def\O{\mathcal{O}} 

\def\F{\mathcal{F}}
\DeclareMathOperator{\PGL}{PGL}

\DeclareMathOperator{\SL}{SL}
\definecolor{miazul}{rgb}{0.60,0.75,0.90}
\definecolor{verdeazul}{rgb}{0.30,0.80,0.30}
\makeatother
\begin{document}     
\title[Stability of foliations]
{On the stability of foliations of degree $3$ with a unique singular point} 
\author[Abel Castorena ]{Abel Castorena }
\address{Abel Castorena Mart\'inez \\
 Centro de Ciencias  Matem\'aticas\\
 Antigua Carretera a P\'atzcuaro \# 8701\\
 Col. Ex Hacienda San Jos\'e de la Huerta\\
 Morelia, Michoac\'an, M\'exico\\
  C.P. 58089.}
\email{abel@matmor.unam.mx}
\author[ Rubí Pantale\'on]{P.~Rub\'i Pantale\'on-Mondrag\'on}
\address{P.~Rub\'i Pantale\'on-Mondrag\'on\\
  Centro de Ciencias  Matem\'aticas\\
 Antigua Carretera a P\'atzcuaro \# 8701\\
  Col. Ex Hacienda San Jos\'e de la Huerta\\
  Morelia, Michoac\'an, M\'exico\\
  C.P. 58089.
 }
\email{petra@matmor.unam.mx}
\author[Juan Vásquez]{Juan Vásquez Aquino}
\address{Juan Vásquez Aquino\\
Unidad académica de matemáticas, Universidad Autónoma de Zacatecas\\
Paseo la Bufa, Av. Solidaridad, Zacatecas, Zacatecas, México\\ C.P. 98060.}
\email{juan.vasquez@cimat.mx}

\thanks{The first author is supported with Grants A1-S-9029, "Moduli
  de curvas y Curvatura en $A_g$" from CONACyT, and PAPIIT IN100723 "Curvas, sistemas lineales en superficies proyectivas y fibrados vectoriales" from DGAPA, UNAM. The second and third authors are supported with a Posdoctoral Fellowship from CONACyT}
\begin{abstract}
Applying Geometric Invariant Theory (GIT), we study the stability of foliations of degree $3$ on $\P^{2}$ with a unique singular point of multiplicity $1,2$, or $3$ and Milnor number $13$. In particular, we characterize those foliations for multiplicity 2 in three cases: stable, strictly semistable, and unstable.
\end{abstract}

\maketitle

\section{Introduction}

In the last decades, the classification of  holomorphic foliations on $\P^2$ of degree $d$ has been of great interest  in Algebraic Geometry. By \cite{CO, GM-K} we know that for $d\geq 2$ a holomorphic foliation with isolated singularities is uniquely determined by its singular scheme.
The case where all singularities are different was studied by \cite{GM-K}, the situation becomes more complicated when we have repeated singular points. In this work, we are interested in those foliations whose singular scheme consists exactly of one point. 

The Geometric Invariant Theory (GIT) introduced by David Mumford, states that given a linear action of a reductive group on a projective variety, it is possible to construct a good quotient if we consider the restricted action on the open set of semistable points by eliminating a closed subset consisting of unstable points of the action. Frances Kirwan shows that it is possible to construct a stratification of the variety by non-singular locally closed subvarieties such that, the unique open stratum is the open subset of semistable points, thus the other strata consist of unstable points. An important problem is to determine the stability of the points of the variety.

The space of foliations of degree $d$ on the complex projective plane $\P^{2}$, which we denote by $\F_{d}$, is a projective space and we define the linear action by change of coordinates of the automorphisms group $\PGL_3(\C)$ of $\P^2$. It is known that the study of singularities, invariant lines, and automorphism groups of foliations can be done up to this action. Moreover, since there is an isogeny between $\PGL_3(\C)$ and $\SL_3(\C)$ we can consider the linear action of $\SL_3(\C)$ instead of $\PGL_3(\C)$, and we know more properties and representations of this reductive group. Naturally, we are interested in the information that we can obtain about the  GIT quotient $\F_d//\SL_3(\C)$. In most cases, this quotient variety is very singular, the worst singularities come from the semistable points whose stabilizers are of positive dimension. What can we say about the stability of the foliations with a unique singular point?  In \cite{AlcantaraGrado1} the GIT quotient of foliations of  degree one on $\P^2$ is constructed showing that $\F_1//\SL_3(\C)\cong \P^1$. The geometric invariant theory of $\F_2$ is studied in \cite{AlcantaraGrado2} and stratification is constructed in the sense of Kirwan. In \cite{CR16} the authors give the stratification of $\F_{3}$ and characterize the foliations in the unstable strata, such stratification is basically constructed by using the diagram of weights of the $\SL_3(\C)$-representation on $\F_3$. In that stratification appear unstable foliations with a unique singular point, although not all of them are explicitly described.

An important invariant  on the foliations  which allows us to characterize them is the multiplicity of the singular point, for example when the multiplicity is one, the local representation of the foliation is well known and determines the singularity type \cite{Brunella}. However, finding examples with specific properties, including when the multiplicity is one is in general very difficult. 

In this work we focus on foliations in $\F_{3}$ with a unique singular point, we determine its type of stability concerning for  to the multiplicity of the singular point, and we give some explicit examples for the stable, semistable non-stable and unstable foliations with a unique singular point. 

The structure of the paper is the following. In Section \ref{Preliminares} we give basic concepts for which this work has developed. In Section \ref{Estabilidad}, we give known results for degree $d\leq 3$. In Section \ref{Results} we give the main results of this work,  and we describe the stability of foliation of degree $3$ with a unique singular point of multiplicity two. We give also conditions on  the defining polynomials of the foliation such that we have a stable, strictly semistable, or unstable foliation. 

\section{Preliminaries}\label{Preliminares}

\subsection{Foliations}

In this section, we give general results of holomorphic
foliations by curves on the complex projective plane. 

Up to nonzero scalar multiplication, a degree $d$ foliation on $\P^{2}$ is defined by a homogeneous vector field 
\[X=P(x,y,z)\frac{\partial}{\partial
       x}+Q(x,y,z)\frac{\partial}{\partial
       y}+R(x,y,z)\frac{\partial}{\partial z},\]

where $P(x,y,z),Q(x,y,z)$ and $R(x,y,z)\in\C[x,y,z]$ are homogeneous polynomials of degree
$d$. Moreover,  for every homogeneous polynomial $G(x,y,z)\in \C[x,y,z]$ of
degree $d-1$, both foliations $X$ and
$X+G(x,y,z)\cdot(x\frac{\partial }{\partial
  x}+y\frac{\partial}{\partial y}+z\frac{\partial}{\partial z})$ define 
the same foliation. We say that the foliations  are equivalent.

Let $\F_{d}$ be the space of foliations on $\P^{2}$ of degree $d$. It is well known that $\F_{d}$ is a
projective space of dimension $d^{2}+4d+2$ (see \cite{GO}). 

\begin{defn}
  A point $p\in \P^{2}$ is a {\bf singular point} of the foliation $X$  if
  $X(p)=\lambda p$ for some $\lambda\in\C$. We denoted by $Sing(X)$ the set of singular points of the foliation $X$.
\end{defn}

If the polynomials that define the foliation have no common factors,
then the foliation has isolated singularities. For a degree greater than one the foliation is defined uniquely by its singular scheme (see \cite{CO}). For the  purposes of this work, we will assume that the  foliations have isolated singularities. Note that we can see the singular locus of a foliation $X$ on $\P^2$ as a subset in $\C^2$ by taking a local representation of $X$ on the open set $U_{0}\subset \P^{2}$ with the vector field \[f(y,z)\frac{\partial}{\partial y}+g(y,z)\frac{\partial}{\partial z}\] where $f(y,z):=Q(1,y,z)-yP(1,y,z)$ and $g(y,z):=R(1,y,z)-zP(1,y,z)$. 

It is easy to see that there is a correspondence between the points of $Sing(X)$ and the points of the variety $V(\langle f,g\rangle)$. Moreover, with this representation, we have two important invariants for a singular point $p$: the Milnor number of $p$ and the multiplicity of $p$, and these invariants are defined as follows:

\begin{defn} Let $p$ be a singular point of a foliation $X$, up to a change of coordinates $p=[1:0:0]$.  
  The {\bf Milnor number} of $p$ is defined by
  \[\mu_{p}(X):=\dim_{\C}\frac{\O_{\C^{2},(0,0)}}{\langle f,g\rangle \cdot\O_{\C^{2},(0,0)}},\]
  where $\O_{\C^{2},(0,0)}$ is the ring of regular functions.

If we write  $f=f_{m}+f_{m+1}+\cdots$, and $g=g_{r}+g_{r+1}+\cdots ,$ as their decomposition into forms, where $f_{m}$ and $g_{r}$ are the forms of lowest degree in $f$ and $g$ respectively, $m,r \in\Z_{\geq 0}$,
the {\bf multiplicity} of $p$ is defined by
     \[ m_{p}(X)=min\{m,r\}.\]
\end{defn}

We can see that the Milnor number of $p$ is the intersection index of the algebraic curves $f$ and $g$ at  the origin, which we will denote by $I_{(0,0)}(f,g)$. If there is no confusion, we denote $I_{(0,0)}(f,g)$ by $I_{0}(f,g)$. 

An important and well known result for foliations with isolated singularities is the following:

\begin{prop}\cite{J06} Let $X$ be a foliation on $\P^{2}$ of degree $d$ with
  isolated singularities
  then \[d^{2}+d+1=\sum_{p\in\P^{2}}\mu_{p}(X).\]
\end{prop}
This result implies that when $X$ is a foliation of degree $3$ with a unique singular point, then the Milnor number is $13$ or equivalently $I_{0}(f,g)=13$.

\subsection{Stratification of a projective variety}

We give a brief overview of Kirwan's techniques to construct a perfectly equivariant stratification of a projective variety $V$ acted on by a reductive group $G$, over the complex numbers, for more details see \cite[Part II]{Kirwan}.

\vskip2mm

Let $V\subset\P^{n}$ be a non-singular complex projective
variety. Consider a reductive group $G$ acting linearly on $V$.
\begin{defn}
  Let $x\in V$ and consider $\bar{x}(\neq 0)$ in the affine cone of
  $V$ such that $\bar{x}\in x$. Denote by $\O(\bar{x})$ the orbit of
  $\bar{x}$, then:
  \begin{enumerate}
  \item $x$ is {\bf unstable} if $0\in\overline{\O(\bar{x})}$. The set of
    unstable points is denoted by $V^{un}$,
  \item $x$ is {\bf semistable} if $0\notin \overline{\O(\bar{x})}$. The set
    of semistable points is denoted by $V^{ss}$,
  \item $x$ is {\bf stable} if it is semistable, the orbit of $x$, $\O(x)$
    is closed in $V^{ss}$ and $\dim \O(x)=\dim G$. The set of stable
    points is denoted by $V^{s}$.
  \end{enumerate}
We say that the point is strictly semistable if it is semistable non-stable.
\end{defn}

It is well known that a quotient variety for $V$ can not exists. Mumford states in \cite{Mumford} that if we consider the open subset of semistable points $V^{ss}=V\setminus V^{un}$ and we restrict the action of $G$ on $V^{ss}$, then there exists a good quotient of this action denoted by $V//G$. Neverteless, classify the points of $V$ according to their stability is an important problem. Kirwan shows that  the unstable points contains information about the quotient variety \cite{Kirwan}.

\begin{defn}
  A finite collection $\{S_{\beta}:\beta \in \mathcal{B}\}$ of subsets of $V$
  forms a {\bf stratification} of $V$, if $V$ is the disjoint union of the
  strata $\{S_{\beta}:\beta\in \mathcal{B}\}$ and there is a partial order
  $\succ$ on the indexing set $\mathcal{B}$ such that
  $\bar{S}_{\beta}\subset\bigcup_{\gamma \geq \beta} S_{\gamma}$
  for every $\beta\in \mathcal{B}$.
\end{defn}

The existence of a stratification is given by the following theorem:

\begin{thm}\cite[Theorem 13.5]{Kirwan} \label{Theorem:Kirwan}
Let $V$ be a non-singular projective
  variety with a linear action by a reductive group $G$. Then there
  exists a stratification $\{S_{\beta}: \beta\in \mathcal{B}\}$ of $V$ such that
  the unique open stratum is $V^{ss}$ and every stratum $S_{\beta}$ in
  the set of unstable points is non-singular locally-closed and
  isomorphic to $G\times_{P_{\beta}}Y^{ss}_{\beta}$, where
  $Y^{ss}_{\beta}$ is a non-singular locally-closed subvariety of $V$
  and $P_{\beta}$ is a parabolic subgroup of $G$.
\end{thm}

The construction of this stratification uses the diagram of weights of the Lie algebra representation of $G$ on $V\subset \P^n$, by considering a maximal torus $T$ of $G$ and the representation of $T$ on
$\C^{n+1}$. This representation splits as a sum of scalar
representations given by characters $\alpha_0,\dots,\alpha_n$ which are weights of $Lie(T)$. 

Since we can to construct a diagram of weights, if we know the weights of the Lie algebra representation we can identify the origin as the weight $0$.  The Hilbert-Mumford numerical criterion states the following:  Let $x=[x_0:\dots:x_n]\in V\subset \P^n$, then $x$ is semistable if the origin $0$ lies in the convex hull of weights $Conv\{\alpha_i\mid x_i\neq 0\}$. Moreover, $x$ is stable if $0$ lies in the interior of $Conv\{\alpha_i\mid x_i\neq 0\}$.

\subsection{Stratification of the space of degree 3 foliations}\label{Stratificacion3}

Let  $\F_{3}$ be the space of degree $3$ foliations and $PGL(3,\C)$ be the automorphism group of $\P^{2}$ acting by change of coordinates on $\F_3$. This action is given by

\[\begin{array}{rcl}
    PGL(3,\C)\times\F_{3} &\rightarrow &\F_{3} \\
    (g,X) &\mapsto & g\cdot X=DgX (g^{-1}).
  \end{array}\]

Since there is an isogeny between $\PGL_3(\C)$ and $\SL_3(\C)$, we can consider the action of $\SL_3(\C)$ instead of $\PGL_3(\C)$. The representation associated with the action $SL(3,\C)$ on $\F_{3} $ is
the kernel of the contraction map:
 \[
   \begin{array}{rccl}
     i_{1,3}: & Sym^{3}(\C^{3})^{*}\otimes\C^{3} & \longrightarrow & Sym^{2}(\C^{3})^{*}\\
              & P(x,y,z)\frac{\partial}{\partial x}+Q(x,y,z)\frac{\partial}{\partial y}+R(x,y,z)\frac{\partial}{\partial z}  &  \mapsto    & \frac{\partial P}{\partial x}+\frac{\partial Q}{\partial y}+\frac{\partial R}{\partial z}.
\end{array}
\]

This kernel is an irreducible representation denoted by  $\Gamma_{3}$, then
\[Sym^{3}(\C^{3})^{*}\otimes\C^{3}=\Gamma_{3}\oplus
  Sym^{2}(\C^{3})^{*},\] thus,
\[\small{\Gamma_{3}=\left\{X=P(x,y,z)\frac{\partial}{\partial
    x}+Q(x,y,z)\frac{\partial}{\partial y}+R(x,y,z)\frac{\partial}{\partial
    z} : \frac{\partial P}{\partial x}+\frac{\partial
    Q}{\partial y}+\frac{\partial R}{\partial z}=0\right\}}.\]

It is easy to see that $\Gamma_3$ is a complex vector space of dimension 24 generated by the following set: 
\[\left\{x^k y^i z^j \frac{\partial}{\partial x}, x^k y^i z^j\frac{\partial}{\partial y} \text{ where } k+i+j=3  \text{ and } x^ky^i\frac{\partial}{\partial z} \text{ with } k+i=3 \right\}.\]

We can apply the Hilbert-Mumford numerical criterion to obtain the explicit coordinates of the unstable foliations. In  \cite{CR16} the authors show that the space of unstable degree 3 foliations on $\P^2$ is a Zariski closed variety with 3 components given by the following vector spaces:

\scalebox{0.9}{\begin{tabular}{l}
$V_1=\langle xy^2\frac{\partial}{\partial x},\,xyz\,\frac{\partial}{\partial x},\,xz^2\frac{\partial}{\partial x},\,y^3\frac{\partial}{\partial x},\,y^2z\frac{\partial}{\partial x},\,yz^2\frac{\partial}{\partial x},\,z^3\frac{\partial}{\partial x},\, xz^2\frac{\partial}{\partial y},\,y^3\frac{\partial}{\partial y},\,y^2z\frac{\partial}{\partial y},\, yz^2\frac{\partial}{\partial y},\,z^3\frac{\partial}{\partial y},\,y^3\frac{\partial}{\partial z}\rangle_{\C},$\\[0.2cm]
$V_2 =\langle x^2z\frac{\partial}{\partial x},\,xyz\frac{\partial}{\partial x},\,xz^2\frac{\partial}{\partial x},\,y^2z\frac{\partial}{\partial x},\,yz^2\frac{\partial}{\partial x},\,z^3\frac{\partial}{\partial x},\,x^2z\frac{\partial}{\partial y},\,xz^2\frac{\partial}{\partial y},\,xyz\frac{\partial}{\partial y},\,y^2z\frac{\partial}{\partial y},\,yz^2\frac{\partial}{\partial y},\,z^3\frac{\partial}{\partial y}\rangle_{\C},$\\[0.2cm]
$V_3 =\langle x^2z\frac{\partial}{\partial x},\, xyz\frac{\partial}{\partial x},\, xz^2\frac{\partial}{\partial x},\, y^3\frac{\partial}{\partial x},\, y^2z\frac{\partial}{\partial x}, \, yz^2\frac{\partial}{\partial x},\, z^3\frac{\partial}{\partial x},\,   xz^2\frac{\partial}{\partial y},\, xyz\frac{\partial}{\partial y},\, y^2z\frac{\partial}{\partial y},\, yz^2\frac{\partial}{\partial y},\, z^3\frac{\partial}{\partial y}  \rangle_{\C}.$\\[0.4cm]
 \end{tabular}}

Thus $\F_3^{un} = \SL_3(\C)\P( V_1)\bigcup \SL_3(\C)\P (V_2)\bigcup \SL_3(\C)\P (V_3)$, where $\SL_3(\C)\P (V_i)$ denotes the $\SL_3(\C)$-orbits of points of the projectivization of the complex vector space $V_i$.
 With the diagram of weights of the Lie algebra representation of $\SL_3(\C)$ on $\F_3$ (see Figure \ref{Fig:Diagrama}) it can be constructed a stratification of $\F_3^{un}$ as in Theorem \ref{Theorem:Kirwan}. This stratification  and  the characterization of some of the strata was given in \cite{CR16}. In this work, we identify some degree 3 unstable foliations with a unique singular point of multiplicity two, and it is important to point out that such unstable foliations were not characterized in the stratification given in \cite{CR16}.

\begin{figure}[h!]
 \centering
 \includegraphics[width=0.6\textwidth]{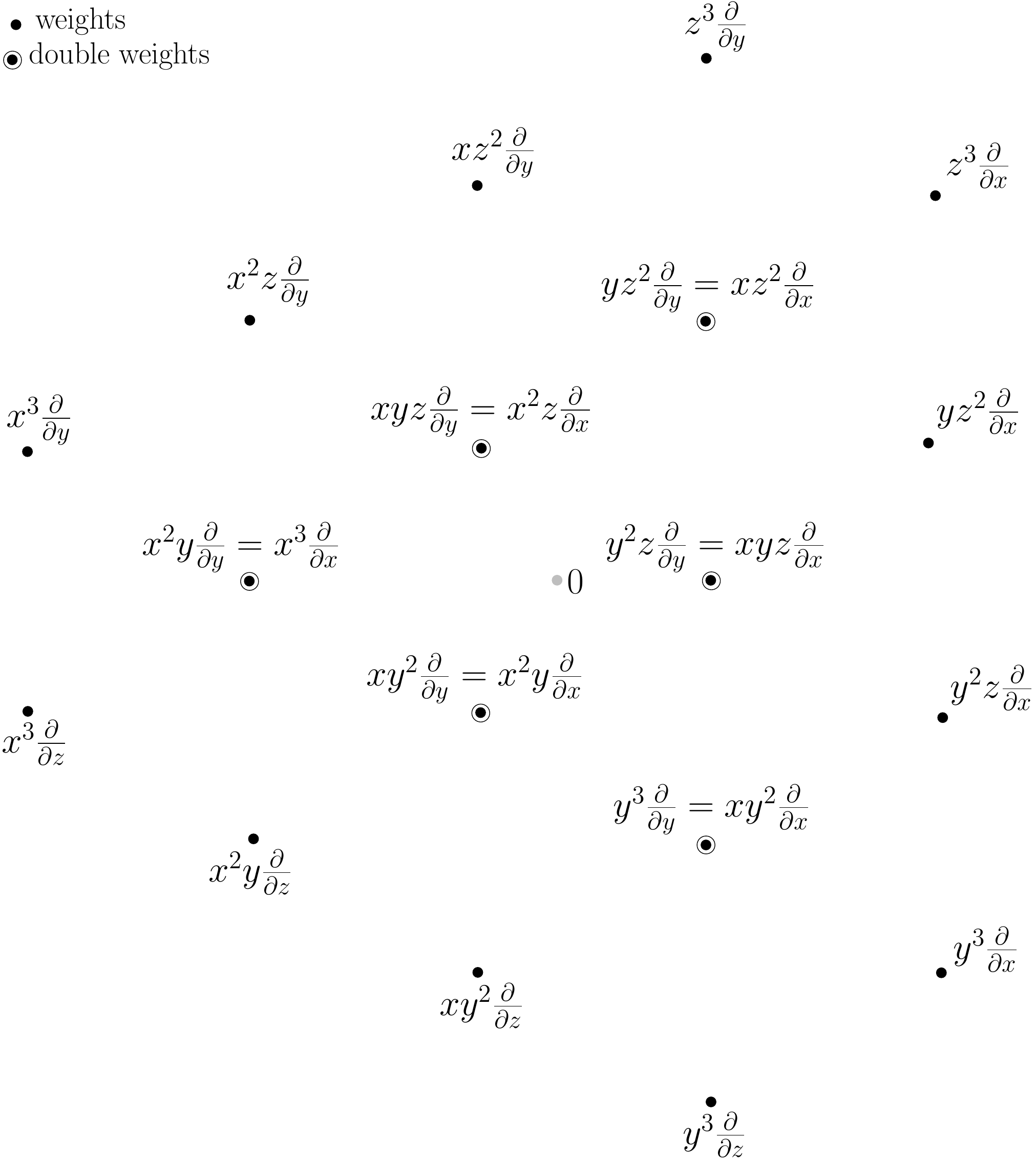}
 \caption{Diagram of weights of the representation on $\F_{3}$} \label{Fig:Diagrama}
 \end{figure}

\section{Stability of foliations with a unique singular point} \label{Estabilidad}

From now on, we will use the following notation:
\begin{align*}
    \mathcal{A}_{d}& :=\{X\in \F_{d}~:~Sing(X)=\{p\} \},\\
    \mathcal{A}_{d,m} & := \{ X\in \mathcal{A}_{d}~:~ m_{p}(X)=m\}.
\end{align*}

We will denote by $\mathcal{A}_{d,m}^{s}$ (resp. $\mathcal{A}_{d,m}^{ss}$, $\mathcal{A}_{d,m}^{un}$) to the set of stable (resp. semistable, unstable) points in  $\mathcal{A}_{d,m}$.

\subsection{Degree 1 foliations}

In \cite{AlcantaraGrado1}, the author shows that a  foliation $X\in \F_1$ is given by $X=A (x\frac{\partial}{\partial x}+y\frac{\partial}{\partial y}+z\frac{\partial}{\partial z})$ module the radial foliation $k(x\frac{\partial}{\partial x}+y\frac{\partial}{\partial y}+z\frac{\partial}{\partial z})$,  where $k\in \C$ and $A$ is a non-zero $3\times 3$ matrix over $\C$. Moreover, the foliation $X$ is unstable if and only if $A$ is nilpotent. 

It is easy to see that if $X\in \mathcal{A}_{1,1}$, then $X$ has the form $(y+az)\frac{\partial}{\partial x}+ bz\frac{\partial}{\partial y}$, thus, in this situation $X$ is unstable. 

\subsection{Degree 2 foliations}

 By \cite{CDGM}, there exist up to change of coordinates $4$ holomorphic foliations in $\mathcal{A}_2$:
\begin{enumerate}
        \item $X_{1}= y^{2}\frac{\partial}{\partial x}+z^{2}\frac{\partial}{\partial y}$,
        \item $X_{2}= -y^{2}\frac{\partial}{\partial x}+(yz-z^{2})\frac{\partial}{\partial y}+z^{2}\frac{\partial }{\partial z}$,
        \item $X_{3}=(y^{2}+z^{2})\frac{\partial}{\partial x}+yz\frac{\partial}{\partial y}$,
        \item $X_{4}= -yz\frac{\partial}{\partial x}+(xy+z^{2})\frac{\partial}{\partial y}+y^{2}\frac{\partial }{\partial z}$.
\end{enumerate}

Constructing the diagram of weights of the Lie algebra representation of $SL_{3}(\C)$ on $\F_{2}$, we can see that $X_{1},X_{2}$, and $X_{3}$ are unstable foliations with a singular point of multiplicity $2$. These foliations were characterized in stratum $3$ of the stratification of the space $\F_{2}$ given in \cite{AlcantaraGrado2}. Similarly, we can see that the foliation $X_{4}$ is an element of $\mathcal{A}_{2,1}^s$. 

\subsection{Degree 3 foliations}
A foliation $X\in\mathcal{A}_{3}$ has multiplicity $m_{p}(X)=1,2$, or $3$. In \cite{EM2011} the authors give a bound that determines when a foliation with isolated singularities is non-stable, in this case, the bound is greater than $1$, thus a foliation with a unique singular point of multiplicity $1$ is stable. We can prove that result as follows.

\begin{prop}
If $X\in \mathcal{A}_{3,1}$ then $X$ is stable.
\end{prop}

\begin{proof}
Consider the general foliation given by \\

\scalebox{.93}{\begin{tabular}{l}
$X=\left(\sum_{i,j}a_{ij}x^{3-i-j}y^iz^j \right)\frac{\partial }{\partial x}+\left(\sum_{i,j}b_{ij}x^{3-i-j}y^iz^j\right)\frac{\partial}{\partial y}+\left(c_{00}x^3+c_{10}x^{2}y+c_{20}xy^{2}+c_{30}y^{3}\right)\frac{\partial}{\partial z}.$ \\[0.4cm]
\end{tabular} }

Without loss of generality, we can choose the singular point $p$ to be  $p=[1:0:0]$, which means that $b_{00}=c_{00}=0$. Moreover,  the points  $[0:1:0]$ and $[0:0:1]$ are not singular points, this implies that $(a_{30},c_{30})\neq 0$ and $(a_{03},b_{03})\neq 0$.  The condition $m_p(X)=1$ implies that $(a_{01},a_{00},b_{10},c_{10})\neq 0$. By considering the diagram of weights (Figure \ref{Fig:Diagrama}) and the Hilbert-Mumford stability criterion we conclude that $X$ is a stable foliation. In particular, $X\in\mathcal{A}_{3,1}$ is stable. 
\end{proof}

\begin{ex}
Consider the foliation \[X=(5xyz-y^{3}+2z^{3})\frac{\partial}{\partial x}+(-\frac{3}{2}x^{2}y-\frac{3}{2}xz^{2}+\frac{9}{2}y^{2}z)\frac{\partial}{\partial y}-3xy^{2}\frac{\partial}{\partial z}.\]
Note that the point $[1:0:0]$ is a singular point of $X$. The local representation around this point is given by $f=-\frac{3}{2}y-\frac{3}{2}z^{2}-\frac{1}{2}y^{2}z+y^{4}-2yz^{3}$ and $g=-3y^{2}-5yz^{2}+y^{3}z-2z^{4}$. Using {\it{Macaulay2}} we can see that the reduced Gr\"obner Basis with respect to the  lexicographic order is $\{z^{13}, y+z^{2}-z^{5}+3z^{11}\}$, this shows that $I_0(f,g)=13$. Then, $[1:0:0]$ is the unique singular point of $X$ with multiplicity one. In fact, this is a particular example of a family of foliations of multiplicity one parameterized by a curve (see \cite{PA}).
\end{ex}

There are a few known examples of foliations with a unique singular point of multiplicity $1$, some of these examples can be found in \cite{AP19, L22, A18, PA}.

For the case of multiplicity $3$, a well known result state that such foliations are unstable. Proof of this fact is given in the following proposition.

\begin{prop} Let
  $X=P(x,y,z)\frac{\partial}{\partial x}+Q(x,y,z)\frac{\partial}{\partial y}+R(x,y,z)\frac{\partial}{\partial z}$ be a foliation in $\F_3$. If $X\in\mathcal{A}_{3,3}$, then $X$ is unstable.
\end{prop}

\begin{proof} Suppose that $p=[1:0:0]$ is the unique singular point of $X$. Since the multiplicity in $p$ is $3$, we have that $X$ should be given by  
\begin{align*}
P(x,y,z)&=a_{20}xy^{2}+a_{11}xyz+a_{02}xz^{2}+a_{30}y^{3}+a_{21}y^{2}z+a_{12}yz^{2}+a_{03}z^{3},\\
Q(x,y,z)&=b_{30}y^{3}+b_{21}y^{2}z+b_{12}yz^{2}+b_{03}z^{3},\\
R(x,y,z)&=c_{30}y^{3}, 
\end{align*}
 with $a_{ij},b_{ij},c_{30}\in\C$. This implies that  $X$ lies in $V_{1}$ (see subsection (\ref{Stratificacion3}) for the definition of $V_1$). Computing the weights of $X$, we see that $X$ is unstable (use Figure \ref{Fig:Diagrama}). Moreover, if $X\in\mathcal{A}_{3,3}^{un}$ then $X$ lies in the stratum $S_{6}$ given in \cite{CR16}.
\end{proof}

\begin{ex}
The foliation $X=(y^{3}+y^{2}z-yz^{2}+z^{3})\frac{\partial}{\partial x}+z^{3}\frac{\partial}{\partial y}$ has a local representation given by  $f=-y^{4}-y^{3}z+y^{2}z^{2}-y\,z^{3}+z^{3}$ and $g= -y^{3}z-y^{2}z^{2}+y\,z^{3}-z^{4}$.

We observe that $p=[1:0:0]$ is the unique singular point of $X$ and the Milnor number $\mu_p(X)= 13$.  This computation can be done in {\it{Macaulay2}} with the following instruction: {\tt degree(ideal(f,g):saturate(ideal(f,g))). }

\end{ex}

\section{Results}\label{Results}

\subsection{Multiplicity $2$}

We want to highlight that in this work we describe some degree 3 unstable foliations with a unique singular point of multiplicity two, which were not characterized in the stratification given in \cite{CR16}, and we  exhibit explicit foliations in $\mathcal{A}_{3,2}^{ss}$ and $\mathcal{A}_{3,2}^s$. 

In all most computations, we will use the polynomials $f(y,z)$ and $g(y,z)$ of  the local representation of a foliation $X$, and we denote by $P=P(1,y,z),Q=Q(1,y,z)$ and $R=R(1,y,z)$ the deshomogenization of $P(x,y,z)$, $Q(x,y,z)$ and $R(x,y,z)$ respectively. 

\subsubsection{Case of Multiplicity $2$; semistable non-stable}

To obtain strictly semistable foliations of degree $3$ we will apply the diagram of weights of the representation in the following way: in the diagram of Figure (\ref{Fig:Diagrama}), consider the line that passes through the origin and at least two weights of the representation. Up to the change of coordinates, we can consider the line that passes through the weights defined by at least one element of the set $\{ x^2z\frac{\partial}{\partial x}, xyz\frac{\partial}{\partial y} \} $, and at least one weight  defined by the set $\{xy^2\frac{\partial}{\partial x}, y^3\frac{\partial}{\partial y}\}$. So, a strictly semistable degree 3 foliations can be written as

\begin{equation}\label{equation:StrictlySemistable}
X=P(x,y,z)\frac{\partial }{\partial x}+Q(x,y,z)\frac{\partial}{\partial y},
\end{equation}
 where  
\begin{equation*}
{\small
\begin{split}
P(x,y,z)&=-c_{02}x^{2}z-c_{21}xy^{2}-c_{12}xyz-b_{03}xz^{2}+a_{30}y^{3}+a_{21}y^{2}z+a_{12}yz^{2}+a_{03}z^{3},\\
Q(x,y,z)&=\left(b_{11}-c_{02}\right)xyz+b_{02}xz^{2}+\left(b_{30}-c_{21}\right)y^{3}+\left(b_{21}-c_{12}\right)y^{2}z+\left(-b_{03}+b_{12}\right)yz^{2}+b_{03}z^{3},\\
\end{split}
}
\end{equation*}
with $a_{ij},b_{ij},c_{ij}\in \C$. These last equations up to foliations equivalence can be also found in \cite[Theorem 9]{CR16}.

With the representation (\ref{equation:StrictlySemistable}), the point $[1:0:0]$ is a singular point of $X$. Moreover, the following conditions hold:

\begin{enumerate}
    \item The point $[0:1:0]$ is a singular point of $X$ if and only if $a_{30}=0$,
    \item The point $[0:0:1]$ is a singular point of $X$ if and only if $a_{03}=b_{03}=0$,
    \item If $a_{30}$ and $b_{30}$ are not zero, then the point $[1:(b_{30}/a_{30}):0]$ is a singular point of $X$.
\end{enumerate}\label{puntosSingulares}

If we consider the point $[1:0:0]$ as the unique singular point of $X$, we can assume that $a_{30}\neq 0$, $(a_{03},b_{03})\neq 0$, and $b_{30}=0$. Thus, the intersection index of the curves defined by the local representation  $f$ and $g$ at the origin $(0,0)$ is 
 \begin{equation}\label{DesFG1}
 I_{0}(f,g)=I_{0}(-a_{30}y^{4},z)+I_{0}(Q,P) = 4+I_0(Q,P).\end{equation}
To get $I_0(f,g)=13$, we need to find the conditions on the coefficients of $P$ and $Q$  such that $I_{0}(Q,P)=9$.

\begin{thm} 
The foliation $X\in \mathcal{A}_{3,2}^{ss}$ if and only if $X=P(x,y,z)\frac{\partial}{\partial x}+Q(x,y,z)\frac{\partial}{\partial y}$, where 
\begin{align*}
P(x,y,z) &=xy^{2}-c_{12}xyz-b_{02}(b_{02}+c_{12})xz^{2}+a_{30}y^{3}+a_{21}y^{2}z+a_{12}yz^{2}+a_{03}z^{3},\\
Q(x,y,z)&=xyz+b_{02}xz^{2}+y^{3}+\left(b_{21}-c_{12}\right)y^{2}z+\left(-b_{02}^{2}-b_{02}c_{12}+b_{12}\right)yz^{2}+b_{02}(b_{02}+c_{12})z^{3},
\end{align*}
where the following relations on the coefficients hold:\\
$a_{30}^{-1}(b_{21}-2\,c_{12}-5\,b_{02})=1$,\\
$a_{21}=-b_{21}c_{12}+c_{12}^{2}-b_{21}b_{02}-7\,b_{02}^{2}+b_{12}$,\\
 $a_{12}=c_{12}^{2}b_{02}+2\,c_{12}b_{02}^{2}-3\,b_{02}^{3}-b_{12}c_{12}-b_{12}b_{02}+c_{12}b_{02}+b_{02}^{2}$,\\
$a_{03}=-b_{02}^{4}-c_{12}^{2}b_{02}-2\,c_{12}b_{02}^{2}-b_{02}^{3}$.
\end{thm}

\begin{proof}

Let $X$ be a foliation in $\mathcal{A}_{3,2}^{ss}$, then $X$ is defined as (\ref{equation:StrictlySemistable}). We can assume that $p=[1:0:0]$ is the unique singular point of $X$. 

Consider the decomposition of  $P$ and $Q$ into  their homogeneous components $P=P_{3}+P_{2}+P_{1}$ and $Q=Q_{3}+Q_{2}$ where
\begin{equation}
\begin{split}
P_{3}&= a_{30}y^{3}+a_{21}y^{2}z+a_{12}yz^{2}+a_{03}z^{3},\\
P_{2}&= -c_{21}y^{2}-c_{12}yz-b_{03}z^{2},\\
P_{1}&= -c_{02}z,\\
Q_{3}&=-c_{21}y^{3}+(b_{21}-c_{12})y^{2}z+(b_{12}-b_{03})yz^{2}+b_{03}z^{3},\\
Q_{2}&=(b_{11}-c_{02})yz+b_{02}z^{2},\\
\end{split}\label{PQ}
\end{equation}

Since $X$ is semistable non-stable, one of the following conditions should be satisfied:

    \[(c_{02}\neq 0 \mbox{ and } c_{21}\neq 0) \mbox{ or } (b_{11}-c_{02}\neq 0 \mbox{ and } c_{21}\neq 0).\]

By these semistability conditions and the uniqueness of the singular point we have $P_{3}\neq 0, Q_{3}\neq 0$ and $P_{2}\neq 0$. We can see that if $Q_{2}=0$, then $I_{0}(P,Q)=3$ and $I_{0}(f,g)=7$. Thus, we assume that $Q_{2}\neq 0$ and we analize two cases: $P_{1}\neq 0$ and $P_{1}=0$. In the first case, $P_{1}\neq 0$ implies $I_{0}(f,g)=7$, then the interesting case to analyze is when $P_1=0$, ($c_{02}=0$). 

Since $Q_{2}\neq 0$, it can be factored by $Q_{2}=zq$, where $q=b_{11}y+b_{02}z$, and note that $z\nmid P_{2}$ because $c_{21}\neq 0$. Also note that if $q\nmid P_{2}$ then $I_{0}(P,Q)=4$. On the other hand, if $q|P_{2}$ this can be factored by $P_{2}=qm$, where $m=y+bz$ for some $b\in\C$ because of $c_{21}\neq 0$.

Define the following polynomial $H:=-zP+mQ=-zP_{3}+mQ_{3}$. This is a degree $4$ homogenous polynomial in two variables, thus $H$ can be factored into four lines as $H=\prod_{i=1}^{4}l_{i}$. Using properties of the intersection index we can deduce that, if $l_{i}\in\{z,q\}$ and $l_{i}\nmid P_{3}$ for all $i$, then 
\[5\leq I_{0}(P,Q)=\sum_{i=1}^{4}I_{0}(l_{i},Q)-I_{0}(z,Q)\leq 9.\]

\begin{claim} The equality $I_0(P,Q)=9$ is obtained if $a_{30}b_{11}c_{21}\neq 0$, $b_{30}=c_{02}=0$, $(a_{03},b_{03})\neq 0$, $H=q^4$, $q|P_{2}$ and $q\nmid P_{3}$.
\end{claim}
To prove this claim note that if $l_i=z$ for some $i$, then $z\mid mQ_3$ which is not possible because $c_{21}\neq 0$. Thus $H=q^4$ and in this case, we will have many conditions on the coefficients.  First note that from $P_2=mq$ we have $b_{11}+c_{21}=0$ and 
$b_{11}b_{03}=b_{02}(b_{02}+c_{12})$. From $H=q^4$,
we have that $b_{11}^4+c_{21}=0$, if we assume  without loss of
generality that $b_{11}=1$, so $c_{21}=-1$  and we also get the following
equations:\\
$a_{30}=b_{21}-2\,c_{12}-5\,b_{02}$,\\
$a_{21}=-b_{21}c_{12}+c_{12}^{2}-b_{21}b_{02}-7\,b_{02}^{2}+b_{12}$,\\
 $a_{12}=c_{12}^{2}b_{02}+2\,c_{12}b_{02}^{2}-3\,b_{02}^{3}-b_{12}c_{12}-b_{12}b_{02}+c_{12}b_{02}+b_{02}^{2}$,\\
$a_{03}=-b_{02}^{4}-c_{12}^{2}b_{02}-2\,c_{12}b_{02}^{2}-b_{02}^{3}$.
\vskip1mm

Now, if $X$ is a foliation satisfying all the above conditions, we can see that the resultant $Res(f,g,z)$ of $f=Q-yP$ and $g=-zP$ is a polynomial in $\C[b_{02},b_{21},b_{12},c_{12}][y]$
 of degree $13$.  Since $f(0,z)=z(b_{03}z+b_{02})$ and  $g(0,z)=0$ we can see that  $f(0,z)=g(0,z)=0$ if and only if $z=0$. Then we have that $(0,0)$ is the unique solution of $f=g=0$ and $I_0(f,g)=13$. \\
It is easy to see by using the diagram of weights (Figure
 \ref{Fig:Diagrama}) that $X$ is a semistable non stable foliation, thus $X\in \mathcal{A}_{3,2}^{ss}$.
\end{proof}

\begin{ex}
Consider the foliation given by
\[X=(x\,y^{2}-5\,y^{3}-2\,x\,y\,z-12\,y^{2}z-3\,x\,z^{2}-y\,z^{2}-10\,z^{3})\frac{\partial}{\partial x}+ (y^{3}+x\,y\,z+2\,y^{2}z+x\,z^{2}+3\,z^{3})\frac{\partial}{\partial y}. \]

A local representation for $X$ is given by 

$f(y,z)=5\,y^{4}+12\,y^{3}z+y^{2}z^{2}+10\,y\,z^{3}+4\,y^{2}z+3\,y\,z^{2}+3\,z^{3}+y\,z+z^{2}$
and 

$g(y,z)= 5\,y^{3}z+12\,y^{2}z^{2}+y\,z^{3}+10\,z^{4}-y^{2}z+2\,y\,z^{2}+3\,z^{3}.$

We can see that the coefficients satisfy the conditions of the above theorem, and with a direct computation in Macaulay2 we can see that $I_0(f,g)=13$. Thus $X\in \mathcal{A}_{3,2}^{ss}$.
\end{ex}

\subsubsection{Case of Multiplicity $2$; stable}
Now we describe a degree 3 stable foliation with a unique singular point of multiplicity 2. 
Let $X=P(x,y,z)\frac{\partial}{\partial x}+Q(x,y,z)\frac{\partial}{\partial y}+
 R(x,y,z)\frac{\partial}{\partial z}$, where
\begin{align*}
P(x,y,z)=& a_{10}x^{2}y+a_{01}x^{2}z+a_{20}xy^{2}+a_{11}xyz+a_{02}xz^{2}+a_{30}y^{3}+a_{21}y^{2}z+a_{12}yz^{2}+a_{03}z^{3},\\
Q(x,y,z)=& b_{30}y^{3}+b_{21}y^{2}z+b_{12}yz^{2}+b_{03}z^{3},\\
R(x,y,z)=& c_{30}y^{3},
\end{align*}
where $a_{10}a_{01}\neq 0$.

\begin{prop}
Let $X$ be the foliation as above. Then $X\in \mathcal{A}_{3,2}^s$ if the coefficients of $P$ and $Q$ satisfy the following conditions:
\begin{enumerate}
\item $b_{30}=-4a_{10}^{3}a_{01}$,
\item $b_{21}=-6a_{10}^{2}a_{01}^{2}$,
\item $b_{12}=-4a_{10}a_{01}^{3}$,
\item $b_{03}=-a_{01}^{4}$,
\item $c_{30}=a_{10}^{4}$,
\item $Res(a_{10}y+a_{01}z, Q-y(a_{20}y^{2}+a_{11}yz+a_{02}z^{2}))=0$,
\item $Res(a_{10}y+a_{01}z,a_{30}y^{3}+a_{21}y^{2}z+a_{12}yz^{2}+a_{03}z^{3})\neq 0$.
\end{enumerate}
\end{prop}

\begin{proof}
Suposse that $[1:0:0]\in Sing(X)$ is the unique singular point, we consider $A=-zQ+yR$. Then $I_{0}(f,g)=I_{0}(A,f)-I_{0}(y,f)$. We can see that $I_{0}(y,f)=I_{0}(y,b_{03}z^{3})=3$ if $b_{03}\neq 0$.

The condition $a_{10}a_{01}\neq 0$ implies that $[1:0:0]$ is of multiplicity $2$.  Since $A$ is a homogeneous polynomial of degree $4$, then $A=\prod_{i=1}^{4}l_{i}$ with $deg(l_{i})=1$. If $l_{i}\notin \{y, a_{10}y+a_{01}z\}$ for some $i=1,2,3,4$, then $I_{0}(l_{i},f)=2$, and $I_{0}(f,g)=5$. 

Assume that $l_{i}\in \{y,a_{10}y+a_{01}z\}$ for some $i$. If we write $f$ as $f=f_{4}+f_{3}-y(a_{10}y+a_{01}z)$, where $f_{4}=y(a_{30}y^{3}+a_{21}y^{2}z+a_{12}yz^{2}+a_{03}z^{3})$, and $f_{3}=Q-y(a_{20}y^{2}+a_{11}yz+a_{02}z^{2})$. Then 
\begin{equation}
I_{0}(l_{i},f)=\left\{
   \begin{array}{lll}
     3 & \mbox{ if }  l_{i}\nmid f_{3},\\
     4 & \mbox{ if } l_{i}| f_{3} \mbox{ and } l_{i}\nmid f_{4}.
    \end{array}
\right.
\end{equation}

If $l_{i}=y$ for some $i$ and $b_{03}\neq 0$ then $y\nmid f$, thus $I_{0}(f,g)\leq 12$. 

We assume that $l_{i}=a_{10}y+a_{01}z$ for all $i=1,2,3,4$, which implies that  $A=(a_{10}y+a_{01}z)^{4}$. Thus  $I_{0}(l_{i},f)=I_{0}(l_{i}, f_{4}+f_{3})$. If $l_{i}\nmid f_{3}$ we have  $I_{0}(f,g)=9$. If $l_{i} |f_{3}$ and $l_{i}\nmid f_{4}$ we deduce that  $I_{0}(l_{i},f)=I_{0}(l_{i},f_{4})=4$. Thus $I_{0}(A,f)=16$ and $I_{0}(f,g)=13$. By the equality $A=(a_{10}y+a_{01}z)^{4}$ we have the conditions of the statement. 

With all these conditions we see that the $0$ lies in the interior of the convex hull defined by the weights of $X$ (see the diagram of weights in Figure \ref{Fig:Diagrama}). Thus $X$ is stable.
\end{proof}

In general, we have no a characterization for the elements in $\mathcal{A}_{3,2}^s$.

\begin{ex}\label{Ejemplo}
Let $X=P(x,y,z)\frac{\partial}{\partial x}+Q(x,y,z)\frac{\partial}{\partial y}
 +R(x,y,z)\frac{\partial}{\partial z}$ a foliation defined by the following polynomials
\begin{align*}
P(x,y,z)=& x^{2}y+ 2x^{2}z-8xy^{2}-16xyz-8xz^{2}+a_{30}y^{3}+a_{21}y^{2}z+a_{12}yz^{2}+a_{03}z^{3},\\
Q(x,y,z)=& -8y^{3}-24y^{2}z-32yz^{2}-16z^{3},\\
R(x,y,z)=&  y^{3}.
\end{align*}
where $a_{30},a_{21},a_{12},a_{03}\in\C^{*}$. Since  $[1:0:0]\in Sing(X)$ we consider  the local representation of $X$. To compute the intersection index $I_0(f,g)$ we consider the polynomial  $A=-zQ+yR$, then $I_{0}(g,f)=I_{0}(A,f)-I_{0}(y,f)$. Since  $A=(y+2z)^{4}$ we have that $I_{0}(A,f)=4I_{0}(y+2z,f)$.

We can write $f$ as $f=-y(a_{30}y^{3}+a_{21}y^{2}z+a_{12}yz^{2}+a_{0,3}z^{3})+(y+2z)(-8z(y+z)-y)$, then $I_{0}(y+2z,f)=4$ and $I_{0}(f,y)=I_{0}(z^3,y)=3$. Thus $I_{0}(f,g)=13$.
\end{ex}

\subsubsection{Case of Multiplicity 2; unstable}  In \cite{CR16} is constructed a stratification of $\F_3$ with 16 strata. In such stratification, the stratum 6 contains the set $\mathcal{A}_{3,3}^{un}$ and it is mentioned that the points of $A_{3,2}^{un}$ can appear in the stratum 15. The following result extends the characterization of the strata 15 given in \cite{CR16}.

\begin{lem} Lex $X$ be the following foliation
\[(b_{11}^{3}y^{3}+3b_{11}^{2}b_{02}y^{2}z+3b_{11}b_{02}^{2}yz^{2}+b_{02}^{3}z^{3})\frac{\partial }{\partial x} + (b_{11}xyz+b_{02}xz^{2}+b_{21}y^{2}z+b_{12}yz^{2}+b_{03}z^{3})\frac{\partial }{\partial y},
\]
where $b_{11}b_{02}\neq 0$ and $Res(b_{11}y+b_{02}z, b_{21}y^{2}+b_{12}yz+b_{03}z^{2})\neq 0$. Then $X\in \mathcal{A}_{3,2}^{un}$.
\end{lem}

\begin{proof}
 
 Let $p=[1:0:0]$ be the unique singular point of $X$. With the local representation $f$, $g$ of $X$ we have that $m_{p}(X)=2$. The intersection index between $f$ and $g$ at the origin  is $I_{0}(f,g)=4+I_{0}(P,Q)$. Since $Q=z(b_{11}y+b_{02}z+b_{21}y^{2}+b_{12}yz+b_{03}z^{2})$ and $P=(b_{11}y+b_{02}z)^{3}$ we have that $I_{0}(P,Q)=3+I_{0}(b_{11}y+b_{02}z,b_{21}y^{2}+b_{12}yz+b_{03}z^{2})=9$. We conclude that $I_{0}(f,g)=13$. Using the diagram of weights (see Figure \ref{Fig:Diagrama}), it is easy to see that with the given conditions, $X$ is an unstable foliation.
\end{proof}

 \bibliographystyle{alpha}
 \bibliography{StabilityFoliationsPetra}

\begin{thebibliography}{GMOB04}

\bibitem[Alc09]{AlcantaraGrado1}
Claudia~R Alc{\'a}ntara.
\newblock The good quotient of the semi-stable foliations of $\mathbb{CP}^{2}$
  of degree 1.
\newblock {\em Results in Mathematics}, 53(1):1--7, 2009.

\bibitem[Alc13]{AlcantaraGrado2}
Claudia~R Alc{\'a}ntara.
\newblock Foliations on $\mathbb{CP}^{2}$ of degree 2 with degenerate
  singularities.
\newblock {\em Bulletin of the Brazilian Mathematical Society, New Series},
  44(3):421--454, 2013.

\bibitem[Alc18]{A18}
Claudia~R Alc{\'a}ntara.
\newblock Foliations on $\mathbb{CP}^{2}$ of degree d with a singular point
  with milnor number $d^{2}+d+1$.
\newblock {\em Revista Matem{\'a}tica Complutense}, 31(1):187--199, 2018.

\bibitem[APM20]{AP19}
Claudia~R Alc{\'a}ntara and Rub{\'\i} Pantale{\'o}n-Mondrag{\'o}n.
\newblock Foliations on $\mathbb{CP}^{2}$ with a unique singular point without
  invariant algebraic curves.
\newblock {\em Geometriae Dedicata}, 207(1):193--200, 2020.

\bibitem[ARL16]{CR16}
Claudia~R Alc\'antara and Ramon Ronzon-Lavie.
\newblock Classification of foliations on $\mathbb{CP}^2$ of degree 3 with
  degenerate singularities.
\newblock {\em Journal of Singularities}, 14:52--73, 2016.

\bibitem[Bru15]{Brunella}
Marco Brunella.
\newblock {\em Birational geometry of foliations}.
\newblock Springer, 2015.

\bibitem[CDBM10]{CDGM}
Dominique Cerveau, Julie D{\'e}serti, D~Garba Belko, and Rafik Meziani.
\newblock G{\'e}om{\'e}trie classique de certains feuilletages de degr{\'e}
  deux.
\newblock {\em Bulletin of the Brazilian Mathematical Society, New Series},
  41(2):161--198, 2010.

\bibitem[CO01]{CO}
Antonio Campillo and Jorge Olivares.
\newblock Polarity with respect to a foliation and {C}ayley-{B}acharach
  theorems.
\newblock {\em Journal fur die Reine und Angewandte Mathematik}, 534:95--118,
  2001.

\bibitem[EM11]{EM2011}
Esteves Eduardo and Marina Marchisio.
\newblock Invariant theory of foliations of the projective plane.
\newblock {\em Rendiconti del Circolo Matematico di Parlerno}, Serie II,
  Suppl.(83):175--188, 2011.

\bibitem[FPR22]{L22}
Percy Fern{\'a}ndez, Liliana Puchuri, and Rudy Rosas.
\newblock Foliations on $\mathbb{CP}^{2}$ with only one singular point.
\newblock {\em Geometriae Dedicata}, 216(5):1--17, 2022.

\bibitem[GMK89]{GM-K}
Xavier G{\'o}mez-Mont and George Kempf.
\newblock Stability of meromorphic vector fields in projective spaces.
\newblock {\em Commentarii Mathematici Helvetici}, 64(1):462--473, 1989.

\bibitem[GMOB04]{GO}
Xavier G{\'o}mez-Mont and Laura Ortiz-Bobadilla.
\newblock Sistemas din{\'a}micos holomorfos en superficies (spanish), volume 3
  of.
\newblock {\em Aportaciones Matem{\'a}ticas: Notas de Investigaci{\'o}n}, 2004.

\bibitem[Jou06]{J06}
Jean-Pierre Jouanolou.
\newblock {\em Equations de Pfaff alg{\'e}briques}, volume 708.
\newblock Springer, 2006.

\bibitem[Kir84]{Kirwan}
Frances~Clare Kirwan.
\newblock {\em Cohomology of quotients in symplectic and algebraic geometry},
  volume~31.
\newblock Princeton University Press, 1984.

\bibitem[MKF94]{Mumford}
David Mumford, Frances Kirwan, and John Fogarty.
\newblock {\em Geometric invariant theory}.
\newblock Springer-Verlag, Berlin, 1994.

\bibitem[PMMdC]{PA}
Petra~Rubí Pantaleón-Mondragón and Abraham Martín~del Campo.
\newblock The {E}uler-{B}etti algorithm to identify foliations in the {H}ilbert
  scheme.
\newblock Preprint at \url{https://doi.org/10.48550/arXiv.2303.17698}.

\end{thebibliography}
\end{document}